\documentclass[reqno]{amsart}

\usepackage{mathabx}
\usepackage[utf8]{inputenc}
\usepackage[english]{babel}
\usepackage{amsmath,amsthm,amsfonts,amssymb,mathtools,mathrsfs}
\usepackage{enumitem}
\usepackage[foot]{amsaddr}
\usepackage[left=1.2in,right=1.2in,top=1.5in,bottom=1.5in]{geometry}
\usepackage[dvipsnames]{xcolor}
\definecolor{dkblue}{RGB}{1,31,91}
\usepackage[numbers,sort&compress]{natbib}
\usepackage[colorlinks=true,pdfstartview=FitV,linkcolor=dkblue,citecolor=dkblue,urlcolor=dkblue]{hyperref}
\hypersetup{
 pdftitle={Finite time singularity formation for a nonlocal Boussinesq system arising in cold-plasma dynamics},
 pdfauthor={Diego Alonso-Oran and Rafael Granero-Belinchon}
}

\normalsize
\theoremstyle{definition}
\newtheorem{Theorem}{Theorem}[section]
\newtheorem{Proposition}[Theorem]{Proposition}
\newtheorem{Lemma}[Theorem]{Lemma}

\newtheorem{Remark}[Theorem]{Remark}
\numberwithin{equation}{section}

\newcommand{\R}{\mathbb R}
\newcommand{\Q}{\mathcal Q}
\newcommand{\Lop}{\mathcal L}
\newcommand{\Nop}{\mathcal N}
\newcommand{\Bop}{\mathcal B}
\newcommand{\Cop}{\mathcal C}
\newcommand{\En}{\mathscr E}
\newcommand{\dd}{\,\mathrm d}
\newcommand{\norm}[1]{\left\lVert#1\right\rVert}
\newcommand{\abs}[1]{\left\lvert#1\right\rvert}

\begin{document}
\allowdisplaybreaks[1]

\keywords{cold plasma, nonlocal Boussinesq system, finite-time singularities}
\subjclass[2020]{35B44, 35Q35}

\title[Finite time singularities for a nonlocal Boussinesq system]{Finite time singularity formation for a nonlocal Boussinesq system arising in cold-plasma dynamics}

\author[D. Alonso-Or\'an]{Diego Alonso-Or\'an}
\address{Departamento de An\'alisis Matem\'atico and Instituto de Matemáticas y Aplicaciones (IMAULL), Universidad de La Laguna, Astrof\'isico Francisco S\'anchez s/n, 38271 La Laguna, Spain. \href{mailto:dalonsoo@ull.edu.es}{dalonsoo@ull.edu.es}}

\author[R. Granero-Belinch\'on]{Rafael Granero-Belinch\'on}
\address{Departamento de Matem\'aticas, Estad\'istica y Computaci\'on, Universidad de Cantabria. Avda. Los Castros s/n, Santander, Spain. \href{mailto:rafael.granero@unican.es}{rafael.granero@unican.es}}

\date{\today}

\begin{abstract}
We study finite-time singularity formation for a nonlocal Boussinesq
system arising in cold-plasma dynamics. We show that smooth,
symmetric initial data with strictly positive total density
and a sufficiently negative velocity slope at the origin
lead to finite-time blow-up. More precisely, we give an explicit upper bound for the lifespan and show that the velocity gradient becomes
unbounded.
\end{abstract}

\thispagestyle{empty}
\maketitle
\tableofcontents

\section{Introduction}\label{sec:introduction}
In this paper, we study the formation of finite-time singularity for the nonlocal Boussinesq
system
\begin{equation}\label{eq:system-intro}
\begin{cases}
 h_t+(hv)_x+v_x=0,\\
 v_t+vv_x+[\Lop,\Nop h]h+\Nop h=0,
\end{cases}
\qquad (t,x)\in[0,T)\times\R,
\end{equation}
supplemented with initial conditions $h(0,x)=h_0(x),\ v(0,x)=v_0(x).$ The nonlocal operators are defined by
\begin{equation}\label{eq:operators-intro}
 \Lop=-\partial_x^2(1-\partial_x^2)^{-1},
 \qquad
 \Nop=\partial_x(1-\partial_x^2)^{-1},
\end{equation}
and
$
 [\Lop,f]g=\Lop(fg)-f\Lop g,
$
denotes the commutator between \(\Lop\) and multiplication by \(f\).
Here \(h\) represents the perturbation of the ionic density from a
constant equilibrium and \(v\) denotes the ionic
velocity. System~\eqref{eq:system-intro} was derived in
\cite{AlonsoOranDuranGranero2024} as an asymptotic model for
collision-free cold plasmas in a magnetic field.

The physical background of \eqref{eq:system-intro} lies in the study of
one-dimensional hydromagnetic waves. Under the cold-plasma approximation,
and neglecting electron inertia, charge separation, and displacement
current, the parent model introduced by Gardner and Morikawa
\cite{GardnerMorikawa1960} takes the form
\begin{equation}\label{eq:parent-plasma}
\begin{cases}
 n_t+(Vn)_x=0,\\
 V_t+VV_x+\dfrac{bb_x}{n}=0,\\[0.4ex]
 b-n-\left(\dfrac{b_x}{n}\right)_x=0.
\end{cases}
\end{equation}
The variables \(n\), \(V\), and \(b\) describe the ionic density, the
ionic velocity, and the magnetic field, respectively. The first two
equations govern the fluid motion, whereas the third relates the magnetic
field to the density through an elliptic equation. This coupling provides
a natural setting in which nonlinear transport interacts with nonlocal
effects.

The derivation and analysis of reduced equations for these waves have a
long history. Gardner and Morikawa \cite{GardnerMorikawa1960} identified
a formal Korteweg--de Vries regime, connecting cold-plasma dynamics with
the theory of weakly nonlinear dispersive waves. Subsequent asymptotic
reductions were developed by Berezin and Karpman
\cite{BerezinKarpman1964}, Kakutani, Ono, Taniuti, and Wei
\cite{KakutaniOnoTaniutiWei1968}, and Su and Gardner
\cite{SuGardner1969}. A rigorous justification of the KdV approximation
on its natural long time scale was obtained by Pu and Li
\cite{PuLi2019}. Related multiscale reductions in magnetohydrodynamics
were studied in \cite{AlonsoOran2021}. These developments motivate the
study of asymptotic models that retain different features of the
underlying dynamics, as well as the comparison of their regularity and
singularity mechanisms.

For the system \eqref{eq:parent-plasma}, local well-posedness was
established by Alonso-Or\'an and Granero-Belinch\'on
\cite{AlonsoOranGranero2024}. More recently, Bae, Choi, and Kwon
\cite{BaeChoiKwon2025} proved finite-time \(C^1\) singularity formation,
including for certain initial data with vanishing velocity gradient.
Singularity formation in a related Euler-Poisson model was investigated
by the same authors in \cite{BaeChoiKwon2024}. These results demonstrate
the relevance of finite-time breakdown in plasma fluid dynamics.
Nevertheless, the passage to an asymptotic model changes the coupling
between the unknowns and the structure of the nonlocal terms.
Consequently, singularity results for the system  \eqref{eq:parent-plasma}, does not
directly settle the corresponding question for
\eqref{eq:system-intro}. \medskip

Finite-time singularity formation has been investigated for several
Boussinesq-type systems. Numerical simulations in
\cite{bona2016singular,SunXieXing2022} provide evidence of finite-time
blow-up, while \cite{BaeGranero2022} establishes a scenario leading
to finite-time singularity formation for the $abcd$-Boussinesq
system. These works motivate the search for conditions on the
initial data that rigorously imply finite-time breakdown.
In this paper, we address this question for the nonlocal
Boussinesq system \eqref{eq:system-intro} arising in cold-plasma
dynamics.

The mathematical analysis of \eqref{eq:system-intro} was initiated
in \cite{AlonsoOranDuranGranero2024}, where local existence and
uniqueness were proved for initial data in
\(H^2(\R)\times H^3(\R)\). That work also identified conservation
laws and a Hamiltonian formulation of the system. Related
traveling-wave solutions were subsequently studied in
\cite{AlonsoOranDuranGranero2025}. Building on the local theory
and the conserved Hamiltonian, we establish finite-time
singularity formation under an explicit condition on the
initial data. The main result can be stated informally as follows.

\begin{Theorem}
There exist smooth initial data 
for which the corresponding solution to \eqref{eq:system-intro}
develops a singularity in finite time.
\end{Theorem}

For the coupled system \eqref{eq:system-intro}, the main difficulty
is the interaction between the density and the velocity. The term
\(vv_x\) can cause a negative velocity slope to become steeper.
At the same time, compression increases the total density \(1+h\),
which enters the nonlocal term in the velocity equation. A proof of
singularity formation must therefore control how the growing density
affects the velocity slope. This coupling prevents a direct application
of the arguments used for the scalar models.

In this paper, we address this difficulty by combining the conserved
Hamiltonian with the positivity of the underlying convolution kernel.
Together with the preservation of symmetry, these properties allow us
to control the nonlocal terms and prove finite-time singularity formation
for symmetric initial data with strictly positive total density and
sufficiently strong compression, cf. Theorem \ref{thm:main}.\medskip

The paper is organized as follows. Section~\ref{sec:notation} introduces the notation and states the main
theorem.
Section~\ref{sec:preliminaries} establishes the preservation of positivity
and symmetry, together with the estimates for the nonlocal terms.
Section~\ref{sec:proof} contains the proof of the main theorem. Finally, Appendix~\ref{app:continuation}
presents the refined energy estimate and proves the
continuation criterion.

\section{Notation and main result}\label{sec:notation}
\subsection{Notation}
 We write \(H^s=H^s(\R)\) to denote the Sobolev space with norm
\[
 \norm{f}_{H^s}^2
 =\int_\R(1+\xi^2)^s\abs{\widehat f(\xi)}^2\dd\xi,
\]
where \(\widehat f\) is the Fourier transform of \(f\). We denote by 
\(f^\pm=\max\{\pm f,0\}\) the positive and negative parts of
a real-valued function. \medskip

Let \(\Q=(1-\partial_x^2)^{-1}\) denote the inverse Helmholtz
operator. Together with \(\Lop\) and \(\Nop\), it admits the
Fourier representation
\begin{equation}\label{eq:operator-definitions}
\begin{gathered}
 \Q=(1-\partial_x^2)^{-1},\qquad
 \Lop=I-\Q,\qquad
 \Nop=\partial_x\Q,\\[0.5ex]
 \widehat{\Q f}(\xi)=\frac{\widehat f(\xi)}{1+\xi^2},
 \qquad
 \widehat{\Lop f}(\xi)=\frac{\xi^2\widehat f(\xi)}{1+\xi^2},
 \qquad
 \widehat{\Nop f}(\xi)=\frac{i\xi\widehat f(\xi)}{1+\xi^2}.
\end{gathered}
\end{equation}
Thus \(\Q\) and \(\Lop\) are self-adjoint on \(L^2(\R)\),
whereas \(\Nop\) is skew-adjoint. All three operators commute
with spatial derivatives, and \(\partial_x\Nop=-\Lop\). The operator \(\Q\) can also be represented as convolution with
the Helmholtz kernel:
\begin{equation}\label{eq:Helmholtz-kernel}
 \Q f=G*f,\qquad G(x)=\frac12e^{-\abs{x}}.
\end{equation}
This kernel satisfies
\[
 G\ge0,\qquad
 \norm{G}_{L^1}=\norm{G'}_{L^1}=1,\qquad
 \abs{G'}=G\quad\text{a.e.}
\]
Writing \(\rho=1+h\) for the total density, we have
\begin{equation}\label{eq:continuity}
 \rho_t+(\rho v)_x=0,\qquad \Q\rho=1+\Q h.
\end{equation}
Finally, we use the nonnegative self-adjoint operators
\(\Bop=\Q^{1/2}\) and \(\sqrt{\Lop}\), defined in Fourier variables by
\[
 \widehat{\Bop f}(\xi)
 =\frac{\widehat f(\xi)}{\sqrt{1+\xi^2}},
 \qquad
 \widehat{\sqrt{\Lop}\,f}(\xi)
 =\frac{\abs{\xi}\widehat f(\xi)}{\sqrt{1+\xi^2}}.
\]

\subsection{Statement of the main result}\label{subsec:main-result}
Before stating the main result, we recall the conserved Hamiltonian
of \eqref{eq:system-intro},
\begin{equation}\label{eq:Hamiltonian-native}
 \En(h,v)=\frac12\int_\R
 \left[(1+h)v^2+(\Bop h)^2+h(\Nop h)^2\right]\dd x,
\end{equation}
introduced in \cite{AlonsoOranDuranGranero2024}.
Its role in the singularity argument becomes clearer when it is
written as
\begin{equation}\label{eq:energy-positive}
 \En(h,v)=\frac12\int_\R
 \left[\rho\bigl(v^2+(\Nop h)^2\bigr)+(\Q h)^2\right]\dd x.
\end{equation}
This identity, proved in Lemma~\ref{lem:Hamiltonian}, shows that
positive total density makes every term in the integral nonnegative. To express the initial condition, we set
\begin{equation}\label{eq:initial-constants}
 \En_0=\En(h_0,v_0),
 \qquad
 K_0=2+\sqrt{2\En_0}.
\end{equation}

\begin{Theorem}\label{thm:main}
Let \((h_0,v_0)\in H^2(\R)\times H^3(\R)\), with \(h_0\) even,
\(v_0\) odd, and \(\rho_0=1+h_0\ge\rho_*>0\). Denote by
\(T_{\max}\) the maximal lifespan of the corresponding
\(H^2\times H^3\) solution \((h,v)\). If \(y_0=-v_0'(0)\) satisfies
\begin{equation}\label{eq:compression-criterion}
 y_0>\sqrt{2K_0\rho_0(0)+\frac{\En_0}{2}},
\end{equation}
then \(\En_0>0\) and
\begin{equation}\label{eq:lifetime-bound}
 T_{\max}\le T_*:=
 \frac1c\log\left(\frac{y_0+c}{y_0-c}\right)<\infty,
 \qquad c=\sqrt{\frac{\En_0}{2}}.
\end{equation}
Moreover, the velocity gradient becomes unbounded:
\begin{equation}\label{eq:slope-norm-breakdown}
 \limsup_{t\uparrow T_{\max}}\norm{v_x(t)}_{L^\infty}=\infty.
\end{equation}
\end{Theorem}

\begin{Remark}\label{rem:local-theory}
The local well-posedness result in
\cite[Section~4]{AlonsoOranDuranGranero2024} provides, for every
\((h_0,v_0)\in H^2(\R)\times H^3(\R)\), a unique maximal solution
satisfying
\begin{equation}\label{eq:solution-class}
 (h,v)\in C\bigl([0,T_{\max});H^2\times H^3\bigr)
 \cap C^1\bigl([0,T_{\max});H^1\times H^2\bigr).
\end{equation}
The time regularity follows directly from the equations and
Sobolev product estimates. In particular, the embeddings
\(H^1(\R)\hookrightarrow C_b^0(\R)\),
\(H^2(\R)\hookrightarrow C_b^1(\R)\), and
\(H^3(\R)\hookrightarrow C_b^2(\R)\) justify the pointwise
identities and the time differentiation of \(\rho(t,0)\) and
\(v_x(t,0)\) used below. For completeness, Appendix~\ref{app:continuation} develops the
symmetrized energy estimates underlying the local theory and
provides the continuation argument, which was not explicitly
addressed in \cite{AlonsoOranDuranGranero2024}. In particular,
if \(T_{\max}<\infty\), then
\begin{equation}\label{eq:continuation-alternative}
 \limsup_{t\uparrow T_{\max}}
 \left(\norm{h(t)}_{H^2}+\norm{v(t)}_{H^3}\right)=\infty.
\end{equation}
Theorem~\ref{thm:continuation} establishes the integral
criterion \eqref{eq:appendix-continuation-criterion}.
\end{Remark}

\begin{Remark}
The theorem establishes finite-time singularity formation, but the
argument does not identify the spatial location of the first singularity.
Moreover, no uniform bound for the density perturbation \(h\)
up to \(T_{\max}\) is asserted, and simultaneous blow-up of
the density is in principle not excluded.
\end{Remark}

\section{Preliminary identities and estimates}\label{sec:preliminaries}
This section collects the preliminary identities and estimates
used in the proof of Theorem~\ref{thm:main}. We include the
computations to make the argument self-contained and easier to
follow. Throughout, \((h,v)\) denotes the maximal solution on
\([0,T_{\max})\) described in Remark~\ref{rem:local-theory}. \medskip

We begin by showing that the even-odd symmetry of the initial
data is preserved by the evolution.

\begin{Lemma}\label{lem:parity}
If \(h_0\) is even and \(v_0\) is odd, then, for every
\(t\in[0,T_{\max})\), \(h(t,\cdot)\) is even and \(v(t,\cdot)\)
is odd. In particular,
\begin{equation}\label{eq:origin-vanishing}
 v(t,0)=h_x(t,0)=(\Nop h)(t,0)=0.
\end{equation}
\end{Lemma}

\begin{proof}
Since \(\Q\) and \(\Lop\) have even Fourier symbols and \(\Nop\)
has an odd symbol, direct substitution shows that $
 \bigl(h(t,-x),-v(t,-x)\bigr)$
also solves \eqref{eq:system-intro}. The assumed parity of
\(h_0\) and \(v_0\) implies that this solution has the same
initial data as \((h,v)\). Uniqueness therefore gives
$
 h(t,-x)=h(t,x),\qquad v(t,-x)=-v(t,x).
$
Finally, \(\Q h\) is even because its convolution kernel is even.
Thus \(v\), \(h_x\), and \(\Nop h=\partial_x\Q h\) are odd
and vanish at \(x=0\).
\end{proof}

We next show that the total density remains positive throughout
the lifespan of the solution, as a consequence of the continuity
equation.

\begin{Lemma}\label{lem:positive-density}
If \(\rho_0=1+h_0\ge\rho_*>0\), then
\(\rho(t,x)>0\) for every \(t<T_{\max}\) and \(x\in\R\).
\end{Lemma}

\begin{proof}
Consider the characteristic flow
\begin{equation}\label{eq:characteristics}
 \partial_tX(t,\xi)=v(t,X(t,\xi)),\qquad X(0,\xi)=\xi.
\end{equation}
For each \(T<T_{\max}\), the solution class
\eqref{eq:solution-class} and Sobolev embedding give
\[
 \sup_{0\le t\le T}
 \left(\norm{v(t)}_{L^\infty}
       +\norm{v_x(t)}_{L^\infty}\right)<\infty.
\]
Thus \(v\) is continuous in time and uniformly Lipschitz in space
on \([0,T]\), so \eqref{eq:characteristics} has a unique solution
for every \(\xi\in\R\). Differentiating
\eqref{eq:characteristics} with respect to \(\xi\), we obtain
\[
 \partial_t(\partial_\xi X)
 =v_x(t,X(t,\xi))\,\partial_\xi X,
 \qquad \partial_\xi X(0,\xi)=1.
\]
Thus, we have that
\[
 \partial_\xi X(t,\xi)
 =\exp\left(\int_0^t v_x(s,X(s,\xi))\dd s\right)>0.
\]
Moreover,
\(\abs{X(t,\xi)-\xi}\le\int_0^t\norm{v(s)}_{L^\infty}\dd s\), so
\(X(t,\cdot)\) maps \(\R\) onto itself. Along this flow, the continuity
equation gives
\[
 \frac{\dd}{\dd t}\rho(t,X(t,\xi))
 =-\rho(t,X(t,\xi))v_x(t,X(t,\xi)).
\]
Hence
\begin{equation}\label{eq:rho-characteristic}
 \rho(t,X(t,\xi))
 =\rho_0(\xi)
 \exp\left(-\int_0^t v_x(s,X(s,\xi))\dd s\right)>0,
\end{equation}
which concludes the proof.
\end{proof}

We next rewrite the nonlocal term in the velocity equation in a
form that will be useful for the estimates below. Set
\[
 \Cop(h)=[\Lop,\Nop h]h+\Nop h.
\]

\begin{Lemma}\label{lem:force}
For every \(h\in H^2(\R)\), we have the identity
\begin{align}
 \Cop(h)
 &=(1+\Q h)\Nop h-\Q(h\Nop h)=(1+\Q h)\Nop h
 -\frac12\Nop\left[(\Q h)^2-(\Nop h)^2\right].
 \label{eq:force-reduction}
\end{align}
\end{Lemma}

\begin{proof}
Since \(\Lop=I-\Q\), expanding the commutator gives
\[
 [\Lop,\Nop h]h
 =-\Q(h\Nop h)+(\Nop h)\Q h,
\]
which proves the first identity. To obtain the second identity, write
\(h=\Q h-\partial_x^2\Q h\) and \(\Nop h=\partial_x\Q h\).
Then
\begin{equation}\label{eq:hux-identity}
\begin{aligned}
 h\Nop h=(\Q h-\partial_x^2\Q h)\,\partial_x\Q h&=\frac12\partial_x
 \left[(\Q h)^2-(\partial_x\Q h)^2\right]\\
 &=\frac12\partial_x
 \left[(\Q h)^2-(\Nop h)^2\right].
\end{aligned}
\end{equation}
Applying \(\Q\) and using \(\Q\partial_x=\Nop\) completes
the proof.
\end{proof}

We next use the conserved Hamiltonian to control \(\Q h\).
The key observation is that, when the total density is positive,
the Hamiltonian can be written as a sum of nonnegative terms.

\begin{Lemma}\label{lem:Hamiltonian}
The Hamiltonian \eqref{eq:Hamiltonian-native} satisfies
\[
 \En(h(t),v(t))=\En_0,\qquad 0\le t<T_{\max},
\]
and admits the representation \eqref{eq:energy-positive}.
In particular, if \(\rho_0\ge\rho_*>0\), then
\begin{equation}\label{eq:Qh-L2-bound}
 \norm{\Q h(t)}_{L^2}^2\le2\En_0,
 \qquad 0\le t<T_{\max}.
\end{equation}
\end{Lemma}

\begin{proof}
Conservation of \(\En\) was established in
\cite{AlonsoOranDuranGranero2024}. To obtain the positive
representation, use \(\Bop^2=\Q\) and
\(h=(1-\partial_x^2)\Q h\). Integration by parts gives
\begin{align*}
 \int_\R(\Bop h)^2\dd x=\int_\R h\,\Q h\dd x
 =\int_\R\left[(\Q h)^2+(\partial_x\Q h)^2\right]\dd x=\int_\R\left[(\Q h)^2+(\Nop h)^2\right]\dd x.
\end{align*}
Substituting into \eqref{eq:Hamiltonian-native}, we find
\[
 \En(h,v)=\frac12\int_\R
 \left[\rho\bigl(v^2+(\Nop h)^2\bigr)+(\Q h)^2\right]\dd x,
\]
which proves \eqref{eq:energy-positive}. If
\(\rho_0\ge\rho_*>0\), Lemma~\ref{lem:positive-density}
ensures that \(\rho>0\) throughout the lifespan. Dropping
the nonnegative weighted terms and using conservation of
\(\En\) yields \eqref{eq:Qh-L2-bound}.
\end{proof}

The positivity of the density and the form of the Helmholtz
kernel allow us to turn \eqref{eq:Qh-L2-bound} into a uniform
bound at the origin. Indeed, comparing the kernels of \(\Q\) and
\(\Q^2\) allows us to apply Cauchy--Schwarz to the quantity \(\Q h\)
controlled by the energy.

\begin{Lemma}\label{lem:kernel-bound}
If \(\rho_0\ge\rho_*>0\), then
\begin{equation}\label{eq:kernel-bound}
 0<1+(\Q h)(t,0)\le K_0,
\end{equation}
for every \(0\le t<T_{\max}\), where \(K_0\) is defined in
\eqref{eq:initial-constants}.
\end{Lemma}

\begin{proof}
By Lemma~\ref{lem:positive-density}, \(\rho>0\).
Since \(G>0\) and \(\Q1=1\), we have
\[
 1+\Q h=\Q\rho=G*\rho>0.
\]
A direct computation gives
\[
 (G*G)(x)=\frac14(1+\abs{x})e^{-\abs{x}}
 \ge\frac12G(x).
\]
Therefore convolving with the positive function \(\rho\), we obtain
\[
 0<\Q\rho\le2\Q^2\rho=2+2\Q(\Q h).
\]
Since \(\norm{G}_{L^2}=1/2\), Cauchy--Schwarz and
\eqref{eq:Qh-L2-bound} yield
\[
 \abs{\Q(\Q h)(t,0)}
 \le\norm{G}_{L^2}\norm{\Q h(t)}_{L^2}
 \le\frac12\sqrt{2\En_0}.
\]
Consequently,
\[
 0<1+(\Q h)(t,0)
 \le2+2\abs{\Q(\Q h)(t,0)}
 \le2+\sqrt{2\En_0}=K_0,
\]
which proves the claim.
\end{proof}

\section{Proof of Theorem~\ref{thm:main}}\label{sec:proof}
The estimates below hold on the maximal interval
\([0,T_{\max})\). Moreover, Lemmas~\ref{lem:parity} and
\ref{lem:positive-density} ensure that the symmetry persists and that
\(\rho>0\). We first note that \(\En_0>0\). Otherwise,
\eqref{eq:energy-positive} and \(\rho_0>0\) would give \(v_0=0\) and
\(\Q h_0=0\), hence \(h_0=0\). This contradicts
\eqref{eq:compression-criterion}. Thus \(c=\sqrt{\En_0/2}\) in
\eqref{eq:lifetime-bound} is strictly positive. \medskip

We divide the proof into several steps:

\noindent\underline{Step 1: the density and velocity slope at the origin.}
\par\nopagebreak[4]
By \eqref{eq:solution-class}, both \(h\) and \(v_x\) belong to
\(C^1([0,T_{\max});H^1(\R))\). Since point evaluation is
continuous on \(H^1(\R)\), the functions
\(t\mapsto\rho(t,0)\) and \(t\mapsto v_x(t,0)\) are \(C^1\),
and their time derivatives can be evaluated directly from
the equations.
Since \(v(t,0)=0\), the density equation
gives
\begin{equation}\label{eq:origin-density}
 \frac{\dd}{\dd t}\rho(t,0)=-\rho(t,0)v_x(t,0).
\end{equation}
Similarly, differentiating the velocity equation in space and then
setting \(x=0\) gives
\begin{equation}\label{eq:origin-slope-exact}
 \frac{\dd}{\dd t}v_x(t,0)
 =-v_x(t,0)^2-\partial_x\Cop(h)(t,0).
\end{equation}
We now estimate the last term. Differentiating the equality \eqref{eq:force-reduction} in Lemma \eqref{lem:force} and using that $\partial_x\Nop h=\Q h-h$, \ $\partial_x\Nop=-\Lop$ we find that
\[
 \partial_x\Cop(h)
 =(\Nop h)^2+(1+\Q h)(\Q h-h)
 +\frac12\Lop\left[(\Q h)^2-(\Nop h)^2\right].
\]
Evaluating at the origin, noticing that \((\Nop h)(t,0)=0\) and using the fact that
\(\Lop=I-\Q\) and \(\Q h-h=1+\Q h-\rho\), we find
\begin{align}
 \partial_x\Cop(h)(t,0)
 &=\bigl(1+(\Q h)(t,0)\bigr)^2
 -\rho(t,0)\bigl(1+(\Q h)(t,0)\bigr)\notag\\
 &\quad+\frac12\bigl((\Q h)(t,0)\bigr)^2
 -\frac12\Q\left[(\Q h)^2-(\Nop h)^2\right](t,0).
 \label{eq:origin-force-exact}
\end{align}
The first and third terms are nonnegative. To estimate the
last term, we use positivity of the Helmholtz kernel and the
energy bound \eqref{eq:Qh-L2-bound}:
\begin{align}
 \Q\left[(\Q h)^2-(\Nop h)^2\right](t,0)
 \le\Q\left[(\Q h)^2\right](t,0)&=\frac12\int_\R e^{-\abs{x}}(\Q h)(t,x)^2\dd x\notag\\
 &\le\frac12\norm{\Q h(t)}_{L^2}^2
 \le\En_0.
 \label{eq:origin-convolution-bound}
\end{align}
Since \(0<1+(\Q h)(t,0)\le K_0\) by Lemma~\ref{lem:kernel-bound}
and \(\rho(t,0)>0\), it follows that
\[
 \partial_x\Cop(h)(t,0)\ge-K_0\rho(t,0)-\frac{\En_0}{2}.
\]
Substituting the previous bound in \eqref{eq:origin-slope-exact} yields
\begin{equation}\label{eq:origin-slope-inequality}
 \frac{\dd}{\dd t}v_x(t,0)
 \le-v_x(t,0)^2+K_0\rho(t,0)+\frac{\En_0}{2}.
\end{equation}
The only term still depending on the evolving density is
\(K_0\rho(t,0)\). The next step controls it by the slope itself.

\medskip
\noindent\underline{Step 2: propagation of the initial inequality.} Set \(y(t)=-v_x(t,0)\). Equations
\eqref{eq:origin-density} and \eqref{eq:origin-slope-inequality}
give
\begin{equation}\label{eq:compression-density-system}
 \frac{\dd}{\dd t}\rho(t,0)=\rho(t,0)y(t),
 \qquad
 y'(t)\ge y(t)^2-K_0\rho(t,0)-c^2.
\end{equation}
The initial condition ensures that
\[
 y_0>0,\qquad y_0^2-2K_0\rho_0(0)-c^2>0.
\]
We show that both inequalities persist. As long as
\(y(t)>0\) and \(y(t)^2-2K_0\rho(t,0)-c^2>0\),
\eqref{eq:compression-density-system} yields
\begin{align}
 \frac{\dd}{\dd t}
 \left[y(t)^2-2K_0\rho(t,0)-c^2\right]
 &=2y(t)y'(t)-2K_0\rho(t,0)y(t)\notag\\
 &\ge2y(t)\left[y(t)^2-2K_0\rho(t,0)-c^2\right]
 >0.
 \label{eq:invariant-differential}
\end{align}
Moreover,
\[
 y'(t)\ge
 \left[y(t)^2-2K_0\rho(t,0)-c^2\right]
 +K_0\rho(t,0)>0.
\]
Thus both terms are increasing as long
as they remain positive. Since both are strictly positive
initially, continuity prevents either from reaching zero
at a time strictly smaller than \(T_{\max}\). Consequently,
\begin{equation}\label{eq:invariant-inequality}
 y(t)\ge y_0>0,\qquad
 y(t)^2>2K_0\rho(t,0)+c^2,
 \qquad 0\le t<T_{\max}.
\end{equation}

\medskip
\noindent\underline{Step 3: the lifespan bound.} The preceding estimate gives
\(K_0\rho(t,0)<\tfrac12(y(t)^2-c^2)\). Substituting into
\eqref{eq:compression-density-system}, we obtain
\begin{equation}\label{eq:Riccati}
 y'(t)\ge\frac12\bigl(y(t)^2-c^2\bigr),
 \qquad y(t)\ge y_0>c.
\end{equation}
To integrate this inequality, set
\[
 z(t)=\frac{y(t)+c}{y(t)-c}>1.
\]
Then \eqref{eq:Riccati} implies
\[
 z'(t)
 =-\frac{2c\,y'(t)}{(y(t)-c)^2}
 \le-c\,\frac{y(t)+c}{y(t)-c}
 =-cz(t).
\]
Using the integrating factor\(e^{ct}\) yields
\begin{equation}\label{eq:time-comparison}
 1<z(t)\le z(0)e^{-ct},
 \qquad 0\le t<T_{\max}.
\end{equation}
Consequently, every \(t<T_{\max}\) satisfies
\[
 t<\frac1c\log z(0)
 =\frac1c\log\left(\frac{y_0+c}{y_0-c}\right)
 =T_*.
\]
It follows that \(T_{\max}\le T_*<\infty\).

\medskip
\noindent\underline{Step 4: blow-up of the velocity gradient.} Since \(T_{\max}<\infty\), the continuation alternative in
Remark~\ref{rem:local-theory} gives \eqref{eq:continuation-alternative}.
Moreover, Theorem~\ref{thm:continuation} yields
\[
 \int_0^{T_{\max}}\norm{v_x(t)}_{L^\infty}\dd t=\infty,
\]
and hence
\[
 \limsup_{t\uparrow T_{\max}}
 \norm{v_x(t)}_{L^\infty}=\infty.
\]
This proves
\eqref{eq:slope-norm-breakdown} and concludes the proof.

\appendix

\section{The continuation criterion}\label{app:continuation}

We use the symmetrized energy underlying the local well-posedness
result in \cite{AlonsoOranDuranGranero2024}, keeping track of the
dependence on \(\norm{h}_{L^\infty}\) and
\(\norm{v_x}_{L^\infty}\). We give the energy calculation below
and then use the continuity equation to obtain a continuation
criterion involving only the velocity gradient. \medskip

First we show a refined a priori energy estimate for the solution. To that purpose, consider the energy
\begin{equation}\label{eq:continuation-energy}
 \mathfrak E(t)
 =\norm{h(t)}_{L^2}^2+\norm{v(t)}_{L^2}^2
 +\norm{\sqrt{\Lop}\,\partial_x^2h(t)}_{L^2}^2
 +\norm{\partial_x^3v(t)}_{L^2}^2.
\end{equation}
We present the a priori estimates for smooth solutions.
The extension to solutions in \eqref{eq:solution-class}
follows by the standard mollification and limiting argument
used in \cite[Section~4]{AlonsoOranDuranGranero2024},
which we do not repeat here.

\begin{Proposition}\label{prop:refined-energy}
For the solution in \eqref{eq:solution-class}, the energy
\(\mathfrak E\) satisfies
\begin{equation}\label{eq:continuation-energy-equivalence}
 \mathfrak E(t)\simeq
 \norm{h(t)}_{H^2}^2+\norm{v(t)}_{H^3}^2.
\end{equation}
Moreover, we have that
\begin{equation}\label{eq:refined-energy-estimate}
 \abs{\mathfrak E'(t)}
 \le C\left(
 1+\norm{h(t)}_{L^\infty}+\norm{v_x(t)}_{L^\infty}
 \right)\mathfrak E(t)
\end{equation}
for almost every \(t<T_{\max}\), where \(C>0\) is independent
of the solution.
\end{Proposition}

\begin{proof}
The norm equivalence
\eqref{eq:continuation-energy-equivalence} follows from
\cite[Lemma~4.1]{AlonsoOranDuranGranero2024}, applied with
\(k=2\), together with the boundedness of
\(\sqrt{\Lop}\) on \(L^2\) and the standard equivalence
\[
 \norm{v}_{H^3}^2
 \simeq \norm{v}_{L^2}^2+\norm{v_{xxx}}_{L^2}^2.
\]
At low order, testing the equations against \(h\) and \(v\),
respectively, and integrating by parts gives
\begin{align*}
 \frac12\frac{\dd}{\dd t}
 \left(\norm{h}_{L^2}^2+\norm{v}_{L^2}^2\right)
 &=-\frac12\int_\R v_xh^2\dd x
   +\int_\R v(h_x-\Nop h)\dd x\\
 &\quad-\int_\R v[\Lop,\Nop h]h\dd x.
\end{align*}
The linear term satisfies
\[
 \left|\int_\R v(h_x-\Nop h)\dd x\right|
 \le C\norm{v}_{L^2}\norm{h}_{H^1}
 \le C\mathfrak E.
\]
Also, using \(\Lop=I-\Q\), we have
\[
 [\Lop,\Nop h]h
 =(\Q h)\Nop h-\Q(h\Nop h).
\]
Since \(\Q f=G*f\) and \(\Nop f=G'*f\), with
\(\norm{G}_{L^1}=\norm{G'}_{L^1}=1\), Young's inequality gives
\[
 \norm{\Q f}_{L^p}\le\norm{f}_{L^p},
 \qquad
 \norm{\Nop f}_{L^p}\le\norm{f}_{L^p},
 \qquad 1\le p\le\infty.
\]
In particular, both operators are bounded on \(L^2\)
and \(L^\infty\). Therefore, it follows that
\begin{align*}
 \norm{[\Lop,\Nop h]h}_{L^2}
 &\le \norm{\Q h}_{L^\infty}\norm{\Nop h}_{L^2}
      +\norm{h}_{L^\infty}\norm{\Nop h}_{L^2}\\
 &\le C\norm{h}_{L^\infty}\norm{h}_{L^2}.
\end{align*}
Consequently,
\begin{equation}\label{eq:low-energy-bound}
\frac{\dd}{\dd t}
 \left(\norm{h}_{L^2}^2+\norm{v}_{L^2}^2\right)
 \le C\left(
 1+\norm{h}_{L^\infty}+\norm{v_x}_{L^\infty}
 \right)\mathfrak E.
\end{equation}

For the high-order terms, differentiate the first equation
twice and the second equation three times, and test against
\(\Lop h_{xx}\) and \(v_{xxx}\), respectively.
The identity
\(\partial_x^3\Nop h=-\Lop h_{xx}\)
makes the linear terms cancel, giving
\begin{equation}\label{eq:top-energy-decomposition}
 \frac12\frac{\dd}{\dd t}
 \left(
 \norm{\sqrt{\Lop}\,h_{xx}}_{L^2}^2
 +\norm{v_{xxx}}_{L^2}^2
 \right)
 =I_1+I_2+I_3,
\end{equation}
where
\begin{align*}
 I_1=-\int_\R\partial_x^3(hv)\Lop h_{xx}\dd x, \quad I_{2}=-\int_\R\partial_x^3(vv_x)v_{xxx}\dd x, \quad I_3=-\int_\R
 \partial_x^3\bigl([\Lop,\Nop h]h\bigr)v_{xxx}\dd x.
\end{align*}
To estimate \(I_1\), we expand $
 \partial_x^3(hv)
 =vh_{xxx}+3v_xh_{xx}+3v_{xx}h_x+v_{xxx}h
$
and write $
 I_1=I_{11}+I_{12}+I_{13}+I_{14}, $
with
\begin{align*}
 I_{11}&=-\int_\R vh_{xxx}\Lop h_{xx}\dd x,\ \ I_{12}=-3\int_\R v_xh_{xx}\Lop h_{xx}\dd x,\\
 I_{13}&=-3\int_\R v_{xx}h_x\Lop h_{xx}\dd x, \ \ I_{14}=-\int_\R v_{xxx}h\Lop h_{xx}\dd x.
\end{align*}
For \(I_{11}\), using
\(h_{xx}=\Q h_{xx}-\partial_x^2\Q h_{xx}\)
and integrating by parts gives
\[
 I_{11}
 =\frac12\int_\R v_x
 \left[
 (\Lop h_{xx})^2-(\Nop h_{xx})^2
 \right]\dd x.
\]
The \(L^2\)-boundedness of \(\Lop\) and \(\Nop\) therefore yields
\[
 \abs{I_{11}}
 \le C\norm{v_x}_{L^\infty}\norm{h_{xx}}_{L^2}^2
 \le C\norm{v_x}_{L^\infty}\mathfrak E.
\]

The terms \(I_{12}\) and \(I_{14}\) are estimated directly:
\begin{align*}
 \abs{I_{12}}
 &\le 3\norm{v_x}_{L^\infty}
       \norm{h_{xx}}_{L^2}\norm{\Lop h_{xx}}_{L^2}
 \le C\norm{v_x}_{L^\infty}\mathfrak E,\\
 \abs{I_{14}}
 &\le \norm{h}_{L^\infty}
       \norm{v_{xxx}}_{L^2}\norm{\Lop h_{xx}}_{L^2}
 \le C\norm{h}_{L^\infty}\mathfrak E.
\end{align*}

For \(I_{13}\), we use the interpolation inequality
\[
 \norm{f_x}_{L^4}^2
 \le 3\norm{f}_{L^\infty}\norm{f_{xx}}_{L^2}.
 \qquad f\in H^2(\R),
\]
Applying this estimate to \(h\) and \(v_x\), followed by
Young's inequality, gives
\begin{align*}
 \abs{I_{13}}
 \le 3\norm{h_x}_{L^4}\norm{v_{xx}}_{L^4}
       \norm{\Lop h_{xx}}_{L^2}&\le C
 \left(\norm{h}_{L^\infty}\norm{v_x}_{L^\infty}\right)^{1/2}
 \norm{h_{xx}}_{L^2}^{1/2}
 \norm{v_{xxx}}_{L^2}^{1/2}
 \norm{\Lop h_{xx}}_{L^2}\\
 &\le C\left(
 \norm{h}_{L^\infty}+\norm{v_x}_{L^\infty}
 \right)\mathfrak E.
\end{align*}

Combining the estimates for \(I_{11},I_{12},I_{13}\), and
\(I_{14}\), we conclude that
\[
 \abs{I_1}
 \le \sum_{j=1}^4\abs{I_{1j}}
 \le C\left(
 \norm{h}_{L^\infty}+\norm{v_x}_{L^\infty}
 \right)\mathfrak E.
\]
For \(I_2\), expanding the derivative gives  we write $I_2=I_{21}+I_{22}+I_{23},$
where
\begin{align*}
 I_{21}=-\int_\R vv_{xxxx}v_{xxx}\dd x, \ \  I_{22}=-4\int_\R v_x(v_{xxx})^2\dd x,\ \ I_{23}&=-3\int_\R v_{xx}^2v_{xxx}\dd x.
\end{align*}
Integration by parts yields
\[
 I_{21}=\frac12\int_\R v_x(v_{xxx})^2\dd x,
 \qquad
 I_{23}=-\int_\R\partial_x\bigl(v_{xx}^3\bigr)\dd x=0.
\]
Combining these identities with the expression for \(I_{22}\),
we obtain
\[
 I_2=-\frac72\int_\R v_x(v_{xxx})^2\dd x,
\]
and hence
\[
 \abs{I_2}
 \le \frac72\norm{v_x}_{L^\infty}\norm{v_{xxx}}_{L^2}^2
 \le C\norm{v_x}_{L^\infty}\mathfrak E.
\]
Finally, using $
 [\Lop,\Nop h]h
 =(\Q h)\Nop h-\Q(h\Nop h),
$
we split the nonlocal term as
$
 I_3=I_{31}+I_{32},
$
with
\begin{align*}
 I_{31}
 =-\int_\R
 \partial_x^3\bigl((\Q h)\Nop h\bigr)v_{xxx}\dd x, \ \ I_{32}
 =\int_\R
 \partial_x^3\Q(h\Nop h)v_{xxx}\dd x.
\end{align*}
The multiplier bounds
\[
 \norm{\Q f}_{H^{s+2}}=\norm{f}_{H^s},
 \qquad
 \norm{\Nop f}_{H^{s+1}}\le\norm{f}_{H^s},
\]
together with the Sobolev product estimates and the
\(L^\infty\)-boundedness of \(\Q\) and \(\Nop\), give
\begin{align*}
 \norm{(\Q h)\Nop h}_{H^3}
 &\le C\left(
 \norm{\Q h}_{L^\infty}\norm{\Nop h}_{H^3}
 +\norm{\Nop h}_{L^\infty}\norm{\Q h}_{H^3}
 \right)\\
 &\le C\norm{h}_{L^\infty}\norm{h}_{H^2}.
\end{align*}
Consequently,
\[
 \abs{I_{31}}
 \le \norm{(\Q h)\Nop h}_{H^3}\norm{v_{xxx}}_{L^2}
 \le C\norm{h}_{L^\infty}\mathfrak E.
\]

Similarly, we have thta
\begin{align*}
 \norm{\Q(h\Nop h)}_{H^3}
 =\norm{h\Nop h}_{H^1}&\le C\left(
 \norm{h}_{L^\infty}\norm{\Nop h}_{H^1}
 +\norm{\Nop h}_{L^\infty}\norm{h}_{H^1}
 \right)\le C\norm{h}_{L^\infty}\norm{h}_{H^1},
\end{align*}
and therefore
\[
 \abs{I_{32}}
 \le \norm{\Q(h\Nop h)}_{H^3}\norm{v_{xxx}}_{L^2}
 \le C\norm{h}_{L^\infty}\mathfrak E.
\]
Thus, we conclude that
\[
 \abs{I_3}
 \le\abs{I_{31}}+\abs{I_{32}}
 \le C\norm{h}_{L^\infty}\mathfrak E.
\]
Combining the bounds for \(I_1\), \(I_2\), and \(I_3\) with
\eqref{eq:low-energy-bound} and
\eqref{eq:top-energy-decomposition} proves
\eqref{eq:refined-energy-estimate} for smooth solutions. For solutions in \eqref{eq:solution-class}, the calculation
is justified by spatial mollification.
The mollifiers commute with the Fourier multipliers, so
the linear cancellation remains exact. The transport
commutators vanish in \(L^2\) by the Friedrichs commutator
estimate, while the remaining products converge by the
stated Sobolev regularity.
\end{proof}

The preceding estimate involves both \(h\) and \(v_x\).
The continuity equation controls the first coefficient
in terms of the time integral of
\(\norm{v_x}_{L^\infty}\), which gives the following criterion.

\begin{Theorem}\label{thm:continuation}
Let \((h,v)\) be the maximal solution in
\eqref{eq:solution-class}. If \(T_{\max}<\infty\), then
\begin{equation}\label{eq:appendix-continuation-criterion}
 \int_0^{T_{\max}}
 \norm{v_x(t)}_{L^\infty}\dd t=\infty.
\end{equation}
In particular,
\begin{equation}\label{eq:appendix-slope-limit}
 \limsup_{t\uparrow T_{\max}}
 \norm{v_x(t)}_{L^\infty}=\infty.
\end{equation}
\end{Theorem}

\begin{proof}
Suppose, to the contrary, that \(T_{\max}<\infty\) and
\begin{equation}\label{eq:Astar-definition}
 A_*:=\int_0^{T_{\max}}
 \norm{v_x(t)}_{L^\infty}\dd t<\infty.
\end{equation}
Along the characteristic flow \eqref{eq:characteristics},
the continuity equation gives
\begin{equation}\label{eq:appendix-rho-flow}
 \rho(t,X(t,\xi))
 =\rho_0(\xi)\exp\left(
 -\int_0^t v_x(s,X(s,\xi))\dd s\right).
\end{equation}
This identity holds without any sign assumption on \(\rho_0\).
For each \(t<T_{\max}\), the flow maps \(\R\) onto itself.
Taking absolute values, we obtain
\begin{equation}\label{eq:appendix-rho-bound}
 \norm{\rho(t)}_{L^\infty}
 \le\norm{\rho_0}_{L^\infty}e^{A_*},
 \qquad
 \norm{h(t)}_{L^\infty}
 \le1+\norm{\rho_0}_{L^\infty}e^{A_*}
 =:M.
\end{equation}

Applying Proposition~\ref{prop:refined-energy} and Gr\"onwall's inequality
shows that
\begin{align}
 \mathfrak E(t)
 &\le\mathfrak E(0)\exp\left(
 C\int_0^t
 \left[
 1+\norm{h(s)}_{L^\infty}
  +\norm{v_x(s)}_{L^\infty}
 \right]\dd s
 \right)
 \notag\\
 &\le\mathfrak E(0)
 \exp\left(C\left[
 (1+ M)t+A_*
 \right]\right).
 \label{eq:appendix-Gronwall}
\end{align}
Since \(T_{\max}<\infty\), the norm equivalence
\eqref{eq:continuation-energy-equivalence} implies
\[
 \sup_{0\le t<T_{\max}}
 \left(
 \norm{h(t)}_{H^2}+\norm{v(t)}_{H^3}
 \right)<\infty.
\]

The local existence theory in
\cite{AlonsoOranDuranGranero2024} provides an existence
time bounded below on bounded subsets of \(H^2\times H^3\).
Indeed, Sobolev embedding in
\eqref{eq:refined-energy-estimate} gives the
norm-dependent bound
\[
 \abs{\mathfrak E'(t)}
 \le C\left(\mathfrak E(t)+\mathfrak E(t)^{3/2}\right),
\]
which also holds uniformly in the regularized local
construction.
The solution can therefore be restarted at any
\(t_0<T_{\max}\) for a common time interval of length
\(\delta>0\), independent of \(t_0\).
Choosing \(t_0\) sufficiently close to \(T_{\max}\) so that
\(t_0+\delta>T_{\max}\), uniqueness allows the restarted
solution to extend the original one beyond \(T_{\max}\).
This contradicts maximality and proves
\eqref{eq:appendix-continuation-criterion}. The proof of \eqref{eq:appendix-slope-limit} follows similarly.
\end{proof}

\section*{Acknowledgements}
D.A.-O. acknowledges partial support from grant RYC2023-045563-I,
funded by
MICIU/\allowbreak{}AEI/\allowbreak{}10.13039/\allowbreak{}501100011033
and ESF+.
D.A.-O. and R.G.-B. also acknowledge support from the projects
``EDPs Deterministas y Estoc\'asticas en Biolog\'ia y Mec\'anica
de Fluidos'' (PID2025-171734NAI00) and
``An\'alisis Matem\'atico Aplicado y Ecuaciones Diferenciales''
(PID2022-141187NB-I00), funded by
MCIN/\allowbreak{}AEI/\allowbreak{}10.13039/\allowbreak{}501100011033/\allowbreak{}FEDER, UE.

\setlength{\bibsep}{1ex}


\begin{thebibliography}{99}

\bibitem{AlonsoOran2021}
D.~Alonso-Or\'an.
\newblock Asymptotic shallow models arising in magnetohydrodynamics.
\newblock \emph{Water Waves}, 3:371--398, 2021.
\newblock \url{https://doi.org/10.1007/s42286-021-00050-4}.

\bibitem{AlonsoOranDuranGranero2024}
D.~Alonso-Or\'an, A.~Dur\'an, and R.~Granero-Belinch\'on.
\newblock Derivation and well-posedness for asymptotic models of cold plasmas.
\newblock \emph{Nonlinear Analysis}, 244:113539, 2024.
\newblock \url{https://doi.org/10.1016/j.na.2024.113539}.

\bibitem{AlonsoOranDuranGranero2025}
D.~Alonso-Or\'an, A.~Dur\'an, and R.~Granero-Belinch\'on.
\newblock On the existence of traveling wave solutions for cold plasmas.
\newblock \emph{Physica D: Nonlinear Phenomena}, 481:134803, 2025.
\newblock \url{https://doi.org/10.1016/j.physd.2025.134803}.

\bibitem{AlonsoOranGranero2024}
D.~Alonso-Or\'an and R.~Granero-Belinch\'on.
\newblock Well-posedness for a hyperbolic--hyperbolic--elliptic system
describing cold plasmas.
\newblock \emph{Applied Mathematics Letters}, 147:108863, 2024.
\newblock \url{https://doi.org/10.1016/j.aml.2023.108863}.

\bibitem{BaeChoiKwon2024}
J.~Bae, J.~Choi, and B.~Kwon.
\newblock Formation of singularities in plasma ion dynamics.
\newblock \emph{Nonlinearity}, 37(4):045011, 2024.
\newblock \url{https://doi.org/10.1088/1361-6544/ad2b16}.

\bibitem{BaeChoiKwon2025}
J.~Bae, J.~Choi, and B.~Kwon.
\newblock Singularity formation of hydromagnetic waves in cold plasma.
\newblock \emph{Applied Mathematics Letters}, 160:109344, 2025.
\newblock \url{https://doi.org/10.1016/j.aml.2024.109344}.

\bibitem{BaeGranero2022}
H. Bae and R. Granero-Belinch{\'o}n,
\textit{Singularity formation for the Serre-Green-Naghdi equations
and applications to $abcd$-Boussinesq systems},
Monatshefte f{\"u}r Mathematik, 198 (2022), 503--516.

\bibitem{BerezinKarpman1964}
Yu.~A. Berezin and V.~I. Karpman.
\newblock Theory of nonstationary finite-amplitude waves in a low-density
plasma.
\newblock \emph{Soviet Physics JETP}, 19:1265--1271, 1964.

\bibitem{bona2016singular}
Jerry~L Bona and Min Chen.
\newblock Singular solutions of a Boussinesq system for water waves.
\newblock {\em J. Math. Study}, 49(3):205--220, 2016.

\bibitem{ChenYang2024}
J.~Chen and S.~Yang.
\newblock Wave breaking phenomenon in the unidirectional non-local wave model.
\newblock \emph{Applied Mathematics Letters}, 150:108949, 2024.
\newblock \url{https://doi.org/10.1016/j.aml.2023.108949}.


\bibitem{GardnerMorikawa1960}
C.~S. Gardner and G.~K. Morikawa.
\newblock \emph{Similarity in the Asymptotic Behavior of Collision-Free
Hydromagnetic Waves and Water Waves}.
\newblock Research Report NYO-9082, Courant Institute of Mathematical
Sciences, New York University, 1960.

\bibitem{KakutaniOnoTaniutiWei1968}
T.~Kakutani, H.~Ono, T.~Taniuti, and C.-C.~Wei.
\newblock Reductive perturbation method in nonlinear wave propagation. II.
Application to hydromagnetic waves in cold plasma.
\newblock \emph{Journal of the Physical Society of Japan}, 24(5):1159--1166,
1968.
\newblock \url{https://doi.org/10.1143/JPSJ.24.1159}.


\bibitem{PuLi2019}
X.~Pu and M.~Li.
\newblock KdV limit of the hydromagnetic waves in cold plasma.
\newblock \emph{Zeitschrift f\"ur angewandte Mathematik und Physik},
70(1):32, 2019.
\newblock \url{https://doi.org/10.1007/s00033-019-1076-4}.

\bibitem{SuGardner1969}
C.~H. Su and C.~S. Gardner.
\newblock Korteweg--de Vries equation and generalizations. III. Derivation of
the Korteweg--de Vries equation and Burgers equation.
\newblock \emph{Journal of Mathematical Physics}, 10(3):536--539, 1969.
\newblock \url{https://doi.org/10.1063/1.1664873}.

\bibitem{SunXieXing2022}
J. Sun, S. Xie and Y. Xing,
\textit{Local discontinuous Galerkin methods for the $abcd$ nonlinear Boussinesq system},
Communications on Applied Mathematics and Computation,
4 (2022), no. 2, 381--416.


\end{thebibliography}
\end{document}